\documentclass[a4paper,11pt]{amsart}
\usepackage[titletoc,toc,title]{appendix}
\usepackage[left=2.7cm,right=2.7cm,top=3.5cm,bottom=3cm]{geometry}
\usepackage{amssymb,latexsym,amsmath,amsthm,amscd}
\usepackage{graphicx}
\usepackage{dsfont}
\usepackage[all]{xy}
\usepackage{mathrsfs}
\usepackage{color}
\definecolor{Blue}{rgb}{0.3,0.3,0.9}

\usepackage[T2A,OT1]{fontenc}
\DeclareSymbolFont{cyrillic}{T2A}{cmr}{m}{n}
\DeclareMathSymbol{\Sha}{\mathalpha}{cyrillic}{216}
\newcommand{\sk}{\vspace{0.1in}}
\newtheorem{thm}{Theorem}[section]
\newtheorem{def-thm}[thm]{Definition-Theorem}

\newtheorem{lem}[thm]{Lemma}
\newtheorem{def-lem}[thm]{Definition-Lemma}
\newtheorem{prop}[thm]{Proposition}

\newtheorem*{ThmA}{Theorem A}

\theoremstyle{definition}
\newtheorem{defn}[thm]{Definition}
\theoremstyle{remark}
\newtheorem*{rem}{Remark}
\numberwithin{thm}{section}
\numberwithin{equation}{section}

\newcommand{\cO}{\mathcal{O}}

\newcommand{\frakl}{\mathfrak{l}}
\newcommand{\pp}{\mathfrak{p}}

\newcommand{\bQ}{\mathbf{Q}}
\newcommand{\bR}{\mathbf{R}}
\newcommand{\bZ}{\mathbf{Z}}
\newcommand{\bC}{\mathbf{C}}

\newcommand{\h}{{H}}

\newcommand{\cR}{\mathbb{I}}

\def\ac{{\rm ac}}

\newcommand{\Ac}{{M}}%{{\mathbf{A}^{\rm ac}}}

\newcommand{\unr}{R_0}

\begin{document}

\title%[A note on the $p$-part of the BSD formula]
{On the $p$-part of the Birch--Swinnerton-Dyer formula %in analytic rank one
for multiplicative primes} %in analytic rank one} %: multiplicative primes}
%{Multiplicative reduction and the $p$-part of the BSD formula in analytic rank one}
%for elliptic curves of analytic rank one}
\author[F.~Castella]{Francesc Castella}
%\address{Department of Mathematics, UCLA, Math Sciences 6363, Los Angeles, 90095-1555 CA, USA}
\address{Mathematics Department, Princeton University, Fine Hall, Princeton,
NJ 08544-1000, USA}
%\email{fccastella@gmail.com}
\email{fcabello@math.princeton.edu}

\thanks{This project has received funding from the European Research Council (ERC) under the European Union's
Horizon 2020 research and innovation programme (grant agreement No. 682152).}


%\keywords{}

%\date{\today}

%\dedicatory{}
%\commby{}

% ----------------------------------------------------------------

\maketitle
%\vspace{-5.mm}
%\begin{center}
%\footnotesize{(with an appendix by XIN WAN)}
%\end{center}

\begin{abstract}
Let $E/\bQ$ be a semistable elliptic curve of analytic rank one, and let $p>3$ be a prime for which
$E[p]$ is irreducible.
%the Galois group of the extension $\bQ(E[p])/\bQ$ is isomorphic to ${\rm GL}_2(\mathbf{F}_p)$.
In this note, following a slight modification of the methods of %Jetchev--Skinner--Wan
\cite{JSW},
we use Iwasawa theory to establish the $p$-part of the Birch and Swinnerton-Dyer formula for $E$.
In particular, we extend the main result of \emph{loc.cit.} %to most primes of
to primes of multiplicative reduction. %under a mild hypothesis.
\end{abstract}

\setcounter{tocdepth}{1}
\tableofcontents

\section{Introduction}

%\subsubsection*{Statement of the main result}

Let $E/\bQ$ be a semistable elliptic curve of conductor $N$, %and minimal discriminant $\Delta$,
and let $L(E,s)$ be the Hasse--Weil $L$-function of $E$. By the celebrated work of Wiles \cite{Fermat-Wiles}
and Taylor--Wiles \cite{TW}, $L(E,s)$ is known to admit analytic continuation to the entire complex plane,
and to satisfy a functional equation relating its values at $s$ and $2-s$.
%For any prime $\ell$, let $\Delta_\ell$ denote the minimal discriminant of $E$ at $\ell$.
The purpose of this note is to prove the following result towards the Birch and Swinnerton-Dyer
formula for $E$. %(\emph{cf.} \cite[Conj.~1.1.1(b)]{JSW}).

\begin{ThmA}
Let $E/\bQ$ be a semistable elliptic curve of conductor $N$ with ${\rm ord}_{s=1}L(E,s)=1$,
and let $p>3$ be a prime such that the mod $p$ Galois representation
\[
\bar{\rho}_{E,p}:{\rm Gal}(\overline{\bQ}/\bQ)\longrightarrow{\rm Aut}_{\mathbf{F}_p}(E[p]) %$\simeq{\rm GL}_2(\mathbf{F}_p)$
\]
is irreducible.
%$E[p]$ is an irreducible ${\rm Gal}(\overline{\bQ}/\bQ)$-module.
%is irreducible.
If $p\mid N$, assume in addition that $E[p]$ is ramified at some prime $q\neq p$. %$p\nmid{\rm ord}_q(\Delta_q)$.
Then
\[
{\rm ord}_p\left(\frac{L'(E,1)}{{\rm Reg}(E/\bQ)\cdot\Omega_E}\right)
={\rm ord}_p\biggl(\#\Sha(E/\bQ)\prod_{\ell\mid N}c_\ell(E/\bQ)\biggr),
\]
where %${\rm Reg}(E/\bQ)$ is the discriminant of the N\'eron--Tate height pairing on $E(\bQ)\otimes_{\bZ}\bR$,
%$\Omega_E$ %=\int_{E(\bR)}\vert\omega_E\vert$
%is the N\'eron period of $E$, $\Sha(E/\bQ)$ is the Tate--Shafarevich group of $E$, and
%$c_\ell(E/\bQ)$ is the Tamagawa number of $E$ at the prime $\ell$.
\begin{itemize}
\item{} ${\rm Reg}(E/\bQ)$ is the discriminant of the N\'eron--Tate height pairing on $E(\bQ)\otimes\bR$;
\item{} $\Omega_E$ %=\int_{E(\bR)}\vert\omega_E\vert$
is the N\'eron period of $E$;
\item{} $\Sha(E/\bQ)$ is the Tate--Shafarevich group of $E$; and
\item{} $c_\ell(E/\bQ)$ is the Tamagawa number of $E$ at the prime $\ell$.
\end{itemize}
In other words, the $p$-part
of the Birch and Swinnerton-Dyer formula holds for $E$.
\end{ThmA}

\begin{rem}
Having square-free conductor, any elliptic curve $E/\bQ$ as in Theorem~A is necessarily non-CM.
By \cite[Thm.~2]{serre-points}, if follows that $\bar{\rho}_{E,p}$ is in fact surjective for all but
finitely many primes $p$; by \cite[Thm.~4]{mazur-isogenies}, this holds as soon as $p\geqslant 11$.
\end{rem}

%Under the further assumption that
When $p$ is a prime of \emph{good} reduction for $E$, Theorem~A
(in the stated level of generality) was first established by Jetchev--Skinner--Wan \cite{JSW}.
(We should note that \cite[Thm.~1.2.1]{JSW} also allows $p=3$ provided $E$ has good supersingular
reduction at $p$, the assumption $a_3(E)=0$ having been removed in a recent work by Sprung; see \cite[Cor.~1.3]{sprung-IMC}.)
For primes $p\mid N$, %(but for $E$ not necessarily semistable),
some particular cases of Theorem~A are contained in the work of Skinner--Zhang (see \cite[Thm.~1.1]{skinner-zhang}) under further
hypotheses on $N$ and, in the case of split multiplicative reduction, on the $L$-invariant of $E$.
Thus the main novelty in Theorem~A is for primes $p\mid N$.
\sk

Similarly as in \cite{JSW}, our proof of Theorem~A uses anticyclotomic Iwasawa theory. In order to
clarify the relation between the arguments in \emph{loc.cit.} and the arguments in this paper, let us recall
that the proof of \cite[Thm.~1.2.1]{JSW} (for primes $p\nmid N$) is naturally divided into two steps:
%(recall that by the work of Gross--Zagier \cite{GZ} and Kolyvagin \cite{kol88}, the assumption on $L(E,s)$ implies that
%${\rm rank}_{\bZ}E(\bQ)=1$ and $\Sha(E/\bQ)<\infty$):

\begin{enumerate}
\item{} \emph{Exact lower bound on the predicted order of $\Sha(E/\bQ)[p^\infty].$}
For this part of the argument, in \cite{JSW} one chooses a suitable imaginary quadratic field $K_1=\bQ(\sqrt{D_1})$ with
$L(E^{D_1},1)\neq 0$; combined with the hypothesis that $E$ has analytic rank one,
it follows that $E(K_1)$ has rank one and that $\#\Sha(E/K_1)<\infty$
by the work of Gross--Zagier and Kolyvagin.
%Letting $N^+$ denote the product of the prime factors of $N$ that are either split or ramified in $K_1$,
The lower bound
\begin{equation}\label{eq:lower}
{\rm ord}_p(\#\Sha(E/K_1)[p^\infty])\geqslant 2\cdot{\rm ord}_p([E(K_1):\bZ.P_{K_1}])
-\sum_{\substack{w\mid N^+\\ \textrm{$w$\;split}}}{\rm ord}_p(c_w(E/K_1)),
\end{equation}
where $P_{K_1}\in E(K_1)$ is a Heegner point, $c_w(E/K_1)$ is the Tamagawa number of $E/K_1$
at $w$, and $N^+$ is the product of the prime factors of $N$ that are either split or ramified in $K_1$,
is then established by combining:
\begin{enumerate}
\item[(1.a)] A Mazur control theorem proved ``\`a la Greenberg'' \cite{greenberg-cetraro}
for an anticyclotomic Selmer group $X_{\rm ac}(E[p^\infty])$ attached to $E/K_1$ (\cite[Thm.~3.3.1]{JSW});
\item[(1.b)] The proof by Xin~Wan \cite{wanIMC}, \cite{wanIMC-SS} of one of
the divisibilities predicted by the Iwasawa--Greenberg Main Conjecture for $X_{\rm ac}^\Sigma(E[p^\infty])$,
namely the divisibility
\[
Ch_{\Lambda_{}}(X_{\rm ac}(E[p^\infty]))\Lambda_{R_0}\subseteq(L_p(f))
\]
where $f=\sum_{n=1}^\infty a_nq^n$ %\in S_2(\Gamma_0(N))$
is the weight 2 newform associated with $E$, $\Lambda_{R_0}$ is a scalar extension of
the anticyclotomic Iwasawa algebra $\Lambda$ for $K_1$,
and $L_p(f)\in\Lambda_R$ is a certain anticyclotomic $p$-adic $L$-function;
%(see the below for the unexplained notations);
\item[(1.c)] The ``$p$-adic Waldspurger formula'' of Bertolini--Darmon--Prasanna \cite{bdp1}
(as extended by Brooks \cite{brooks} %---see also \cite{LZZ}---
to indefinite Shimura curves):
\[
L_p(f,\mathds{1})=(1-a_pp^{-1}+p^{-1})^2\cdot(\log_{\omega_E}P_{K_1})^2
\]
relating the value of $L_p(f)$ at the trivial character
%(which lies outside the range of classical interpolation)
to the formal group logarithm of the Heegner point $P_{K_1}$.
\end{enumerate}
When combined with the known $p$-part of the Birch and Swinnerton-Dyer formula for
the quadratic twist $E^{D_1}/\bQ$
%(by \cite{SU} and \cite{wan-kobayashi},
%since $L(E^{D_1}/\bQ,1)\neq 0$),
(being of rank analytic zero, this follows from %the cases of the Iwasawa main conjecture for ${\rm GL}_2$ estabilished in
\cite{SU} and \cite{wan-kobayashi}), inequality $(\ref{eq:lower})$ easily yields the exact lower bound for $\#\Sha(E/\bQ)[p^\infty]$
predicted by the BSD conjecture.

\item{} \emph{Exact upper bound on the predicted order of $\Sha(E/\bQ)[p^\infty]$.}
For this second part of the argument, in \cite{JSW} one chooses another imaginary quadratic field $K_2=\bQ(\sqrt{D_2})$
(in general different from $K_1$) such that $L(E^{D_2},1)\neq 0$. Crucially, $K_2$
is chosen so that the associated $N^+$
(the product of the prime factors of $N$ that are split or ramified in $K_2$)
is \emph{as small as possible} in a certain sense; this ensures optimality of
the upper bound provided by Kolyvagin's methods:
\begin{equation}\label{eq:upper-bound}
{\rm ord}_p(\#\Sha(E/K_2)[p^\infty])\leqslant 2\cdot{\rm ord}([E(K_2):\bZ.P_{K_2}]),
\end{equation}
where $P_{K_2}\in E(K_2)$ is a Heegner point coming from a parametrization of $E$
by a Shimura curve attached to an indefinite quaternion algebra
(which is nonsplit unless $N$ is prime).
Combined with the %general Gross--Zagier formula \cite{GZ}, \cite{YZZ}
general Gross--Zagier formula \cite{YZZ} %\cite{CSY}
and the $p$-part of the Birch and Swinnerton-Dyer formula for $E^{D_2}/\bQ$, inequality $(\ref{eq:upper-bound})$ then yields
the predicted optimal upper bound for $\#\Sha(E/\bQ)[p^\infty]$.
\end{enumerate}

Our proof of Theorem~A dispenses with part (2) of the above
argument; %\footnote{As will soon be clear to the reader, Kolyvagin's method of Euler systems are
%still embedded in our argument, however.};
in particular, it only requires the use of classical modular parametrizations of $E$.
Indeed, if $K$ is an imaginary quadratic field satisfying the following %Heegner
hypotheses relative to the square-free integer $N$:
\[
\begin{split}
\bullet\; & %\textrm{$N=N^+N^-$ with $(N^+,N^-)=1$};\\
\textrm{every prime factor of $N$ is either split or ramified in $K$};\\
\bullet\; & %\textrm{$\ell\mid N^+$ if and only if $\ell$ splits in $K$};\\
\textrm{there is at least one prime $q\mid N$ nonsplit in $K$};\\
\bullet\; & %\textrm{$\ell\mid N^-$ if and only if $\ell$ ramifies in $K$};\\
\textrm{$p$ splits in $K$},
%\bullet\; & \textrm{$N^->1$ has an even number of prime factors},
\end{split}%\tag{Heeg}
\]
%(note that, except for a slight change of notation, this is a special case of hypothesis (gen-H) in \cite{JSW}),
in \cite{cas-BF} (for good ordinary $p$) and \cite{cas-wan-ss} (for good supersingular $p$) we have completed
under mild hypotheses the proof of the Iwasawa--Greenberg main conjecture
for the associated $X_{\rm ac}(E[p^\infty])$:  %(including $p\nmid N$);
\begin{equation}\label{eq:IMC}
Ch_{\Lambda_{}}(X_{\rm ac}(E[p^\infty]))\Lambda_{R_0}=(L_p(f)).
\end{equation}
With this result at hand, a simplified form (since $N^-=1$ here)
of the arguments from \cite{JSW}
 %\footnote{Since $N^-=1$ here, thus avoiding e.g. the comparison of periods in [\emph{loc.cit}, $\S{4.5}$].}
in part (1) above lead
to an \emph{equality} in $(\ref{eq:lower})$ taking $K_1=K$,
and so to the predicted order of $\Sha(E/\bQ)[p^\infty]$ when $p\nmid N$.
\sk

%On the other hand,
To treat the primes $p\mid N$
of multiplicative reduction for $E$ (which, as already noted, is the only
new content of Theorem~A),
we use Hida theory. Indeed, if $a_p$ is the $U_p$-eigenvalue of $f$ for such $p$,
we know that $a_p\in\{\pm 1\}$, so in particular $f$ is ordinary at $p$.
Let $\mathbf{f}\in\cR[[q]]$ be the Hida family associated with $f$,
where $\cR$ is a certain finite flat extension of the one-variable
Iwasawa algebra.
%Building upon work of Fouquet \cite{Fouquet}, %Wan \cite{wanIMC}, and
%Kings--Loeffler--Zerbes \cite{KLZ2}, %Ochiai \cite{Ochiai-MC},
%and Wan \cite{wanIMC} among others,
In Section~\ref{sec:MC}, we deduce from %our ealier work
\cite{cas-BF} and %the divisibility in
\cite{wanIMC} a proof %under mild hypotheses
of a two-variable analog of the Iwasawa--Greenberg main conjecture $(\ref{eq:IMC})$ over the Hida family:
\[
Ch_{\Lambda_{\cR}}(X_{\rm ac}(A_{\mathbf{f}}))\Lambda_{\cR,R_0}=(L_p(\mathbf{f})),
\]
where $L_p(\mathbf{f})\in\Lambda_{\cR,R_0}$ is the two-variable anticyclotomic $p$-adic
$L$-function introduced in \cite{cas-2var}.
%relating the characteristic ideal of a torsion module $X_{\rm ac}(A_{\mathbf{f}})$
%over the two-variable anticyclotomic Iwasawa algebra $\Lambda_{\cR}$ for $K$ over $\cR$, to the ideal
%generated by a certain two-variable anticyclotomic $p$-adic $L$-function introduced in \cite{cas-2var}.
By construction, $L_p(\mathbf{f})$ specializes to $L_p(f)$
in weight $2$, %(as naturally extended in \cite{cas-split} to the multiplicative setting),
and by a control theorem for the Hida variable, %(see Section~\ref{sec:Selmer}),
the characteristic ideal of $X_{\rm ac}(A_{\mathbf{f}})$ similarly specializes to $Ch_\Lambda(X_{\rm ac}(E[p^\infty]))$,
yielding a proof of the Iwasawa--Greenberg main conjecture $(\ref{eq:IMC})$ in the multiplicative
reduction case. Combined with the anticyclotomic control theorem of (1.a)
%(which applies in greater generality than explicitly exploited in \cite{JSW}),
and the natural generalization (contained in \cite{cas-split})
of the $p$-adic Waldspurger formula in (1.c) to this case: %(see \cite{cas-split})
\[
L_p(f,\mathds{1})=(1-a_pp^{-1})^2\cdot(\log_{\omega_E}P_{K})^2,
\]
%this allows us to
we arrive at the predicted formula for $\#\Sha(E/\bQ)[p^\infty]$
just as in the good reduction case.
\sk

%(Note that three aforementioned references exploit the Euler system of Heegner points
%and big Heegner points \cite{howard-invmath}, and so Kolyvagin's methods are still embedded in our approach.)

\noindent\emph{Acknowledgements.} As will be clear to the reader, this note borrows many ideas
and arguments from \cite{JSW}. %, and perhaps it would most naturally be seen as an addendum to their beautiful paper.
It is a pleasure to thank Chris Skinner %and Xin Wan
for several useful conversations.

\section{Selmer groups}\label{sec:Selmer}

\subsection{Definitions}\label{sec:defs}

Let $E/\bQ$ be a semistable elliptic curve of conductor $N$, and let $p\geqslant 5$ be a prime such that
the mod $p$ Galois representations
\[
\bar{\rho}_{E,p}:{\rm Gal}(\overline{\bQ}/\bQ)\longrightarrow{\rm Aut}_{\mathbf{F}_p}(E[p])
\]
is irreducible. Let $T=T_p(E)$ be the $p$-adic Tate module of $E$, and set $V=T\otimes_{\bZ_p}\bQ_p$.

Let $K$ be an imaginary quadratic field in which $p=\pp\overline{\pp}$ splits,
and for every place $w$ of $K$ define the \emph{anticyclotomic local condition}
$\h^1_{\ac}(K_w,V)\subseteq\h^1(K_w,V)$ by
\[
\h^1_{\ac}(K_w,V):=\left\{
\begin{array}{ll}
\h^1(K_{\overline{\pp}},V)&\textrm{if $w=\overline{\pp}$;}\\
0&\textrm{if $w=\pp$;}\\
\h^1_{\rm ur}(K_w,V)&\textrm{if $w\nmid p$},
\end{array}
\right.
\]
where $\h^1_{\rm ur}(K_w,V):={\rm ker}\{\h^1(K_w,V)\rightarrow\h^1(I_w,V)\}$
is the unramified part of cohomology.

\begin{defn}
The \emph{anticyclotomic Selmer group} for $E$ is
\[
H^1_{\rm ac}(K,E[p^\infty])
:={\rm ker}\biggl\{\h^1(K,E[p^\infty])\longrightarrow\prod_w\frac{\h^1(K_w,E[p^\infty])}{\h^1_{\ac}(K_w,E[p^\infty])}\biggr\},
\]
where $\h^1_{\ac}(K_w,E[p^\infty])\subseteq \h^1(K_w,E[p^\infty])$ is the image of $\h^1_{\ac}(K_w,V)$ under the natural map
$\h^1(K_w,V)\rightarrow\h^1(K_w,V/T)\simeq\h^1(K_w,E[p^\infty])$.
\end{defn}

%\subsubsection{Selmer groups at infinite level}

Let $\Gamma={\rm Gal}(K_\infty/K)$ be the Galois group of the anticyclotomic
$\bZ_p$-extension of $K$, and let $\Lambda=\bZ_p[[\Gamma]]$ be
the anticyclotomic Iwasawa algebra. Consider the $\Lambda$-module
\[
\Ac:=T\otimes_{\bZ_p}\Lambda^*,
\]
where $\Lambda^*={\rm Hom}_{\rm cont}(\Lambda,\bQ_p/\bZ_p)$ is
the Pontryagin dual of $\Lambda^\ac$. Letting $\rho_{E,p}$ denote the
natural action of $G_K:={\rm Gal}(\overline{\bQ}/K)$ on $T$, the $G_K$-action on
$\Ac$ is given by $\rho_{E,p}\otimes\Psi^{-1}$, where $\Psi$ is the
composite character $G_K\twoheadrightarrow\Gamma\hookrightarrow\Lambda^\times$.

\begin{defn}\label{def:ac-Sel}
The \emph{anticyclotomic Selmer group} for $E$ over $K_\infty^{\rm ac}/K$ is defined by
\[
{\rm Sel}_{\pp}(K_\infty,E[p^\infty]):={\rm ker}\biggl\{H^1(K,\Ac)\longrightarrow
H^1(K_\pp,\Ac)\oplus\prod_{w\nmid p}H^1(K_w,\Ac)\biggr\}.
\]
\end{defn}
%and the corresponding \emph{Greenberg Selmer group} by
%\[
%\mathfrak{Sel}_\pp(E[p^\infty]):={\rm ker}\biggl\{H^1(K,\Ac)\longrightarrow
%H^1(K_\pp,\Ac)\oplus\bigoplus_{w\nmid p}H^1(I_w,\Ac)\biggr\}.
%\]
More generally, for any given finite set $\Sigma$ of places $w\nmid p$ of $K$, define the
``$\Sigma$-imprimitive'' Selmer group ${\rm Sel}_\pp^\Sigma(K_\infty,E[p^\infty])$ %and $\mathfrak{Sel}_\pp^\Sigma(E[p^\infty])$
by dropping the summands $H^1(K_w,\Ac)$ %and $H^1(I_w,\Ac)$ for $w\in\Sigma$, respectively,
for the places $w\in\Sigma$ in the above definition. Set
\[
X_{\ac}^\Sigma(E[p^\infty]):={\rm Hom}_{\bZ_p}({\rm Sel}_\pp^\Sigma(K_\infty,E[p^\infty]),\bQ_p/\bZ_p),
\]
which is easily shown to be a finitely generated $\Lambda$-module.

%Note that
%\[
%{\rm Sel}_\pp(E[p^\infty])=\varinjlim_n H^1_{\ac}(K_n^{\ac},E[p^\infty])
%\]
%by Shapiro's Lemma.

\subsection{Control theorems}\label{sec:control}

Let $E$, $p$, and $K$ be an in the preceding section, and let
$N^+$ denote the product of the prime factors of $N$ which are split in $K$.

\subsubsection*{Anticlotomic Control Theorem}

Denote by $\hat{E}$ the formal group of $E$, and let
\[
\log_{\omega_E}:E(\bQ_p)\longrightarrow\bZ_p
\]
the formal group logarithm attached to a fixed invariant differential $\omega_E$
on $\hat{E}$. Letting $\gamma\in\Gamma$ be a fixed topological generator, we identify the one-variable
power series ring $\bZ_p[[T]]$ with the Iwasawa algebra $\Lambda=\bZ_p[[\Gamma]]$ by sending
$1+T\mapsto\gamma$.

\begin{thm}\label{thm:control-greenberg}
Let $\Sigma$ be any set of places of $K$ not dividing $p$, and
assume that ${\rm rank}_{\bZ}(E(K))=1$ and that $\#\Sha(E/K)[p^\infty]<\infty$.
%the following hold:
%\begin{itemize}
%\item[(i)]{} ${\rm corank}_{\bZ_p}{\rm Sel}(K,E[p^\infty])=1$,
%\item[(ii)]{}
%The restriction map
%\[
%{\rm Sel}(K,E[p^\infty])\overset{{\rm loc}_w}\longrightarrow\delta_w(E(K_w)\otimes\bQ_p/\bZ_p)
%\]
%is surjective for all $w\mid p$.
%\end{itemize}
Then $X_{\rm ac}^\Sigma(E[p^\infty])$ is $\Lambda$-torsion, and
letting $f_{\rm ac}^\Sigma(T)\in\Lambda$ be a generator of $Ch_{\Lambda}(X_{\rm ac}^\Sigma(E[p^\infty]))$,
we have
\begin{align*}
\#\bZ_p/f_{\rm ac}^\Sigma(0)&=\#\Sha(E/K)[p^\infty]\cdot
\Biggl(\frac{\#\bZ_p/((1-a_pp^{-1}+\varepsilon_p)\log_{\omega_E}P)}{[E(K)\otimes\bZ_p:\bZ_p.P]}\Biggr)^2\\
&\quad\times\prod_{\substack{w\mid N^+\\w\not\in\Sigma}}c^{(p)}_w(E/K)\cdot\prod_{w\in\Sigma}\#H^1(K_w,E[p^\infty]),
\end{align*}
where $\varepsilon_p=p^{-1}$ if $p\nmid N$ and $\varepsilon_p=0$ otherwise,
$P\in E(K)$ is any point of infinite order,
and $c_w^{(p)}(E/K)$ is the $p$-part of the Tamagawa number of $E/K$ at $w$.
\end{thm}

\begin{proof}
As we are going to show, this follows easily from the ``Anticyclotomic Control Theorem'' established in
\cite[\S{3.3}]{JSW}. The hypotheses imply that ${\rm corank}_{\bZ_p}{\rm Sel}(K,E[p^\infty])=1$
and that the natural map
\[
E(K)\otimes\bQ_p/\bZ_p\longrightarrow E(K_w)\otimes\bQ_p/\bZ_p
\]
is surjective for all $w\mid p$. By \cite[Prop.~3.2.1]{JSW} (see also the discussion in [\emph{loc.cit.}, p.~22])
it follows that $H^1_{\ac}(K,E[p^\infty])$ is finite with
\begin{equation}\label{eq:321}
\#H^1_{\ac}(K,E[p^\infty])=\#\Sha(E/K)[p^\infty]\cdot
\frac{[E(K_\pp)_{/{\rm tors}}\otimes_{}\bZ_p:\bZ_p.P]^2}{[E(K)\otimes\bZ_p:\bZ_p.P]^2},
\end{equation}
where $E(K_\pp)_{/{\rm tors}}$ %=E(K_\pp)/E(K_\pp)_{\rm tors}$
is the quotient $E(K_\pp)$ by its maximal torsion submodule,
and $P\in E(K)$ is any point of infinite order. If $p\nmid N$, then
\begin{equation}\label{eq:p-notdiv}
[E(K_\pp)_{/{\rm tors}}\otimes\bZ_p:\bZ_p.P]=\frac{\#\bZ_p/((\frac{1-a_p+p}{p})\log_{\omega_E}P)}{\#H^0(K_\pp,E[p^\infty])}
\end{equation}
as shown in \cite[p.~23]{JSW}, and substituting $(\ref{eq:p-notdiv})$ into $(\ref{eq:321})$
we arrive at
\[
\#H^1_{\ac}(K,E[p^\infty])=\#\Sha(E/K)[p^\infty]\cdot
\Biggl(\frac{\#\bZ_p/((\frac{1-a_p+p}{p})\log_{\omega_E}P)}{[E(K)\otimes\bZ_p:\bZ_p.P]\cdot\#H^0(K_\pp,E[p^\infty])}\Biggr)^2,
\]
from where the result follows immediately by \cite[Thm.~3.3.1]{JSW}.

%On the other hand, if $p\mid N$ then %$E$ is a Tate curve and so
%there is a $G_{\bQ_p}$-equivariant isomorphism
%\begin{equation}\label{eq:Tate}
%\Phi_E:\overline{\bQ}_p^\times/q_E^{\bZ}\overset{\simeq}\longrightarrow E(\overline{\bQ}_p)
%\end{equation}
%for some $q_E\in p\bZ_p$ inducing an isomorphism of formal groups
%$\hat{\mathbf{G}}_m\simeq\hat{E}$. Letting $\log_{q_E}:\bQ_p^\times\rightarrow\bZ_p$
%be the branch of the $p$-adic logarithm satisfying $\log_{q_E}(q_E)=0$, one
%easily checks (see e.g. \cite[\S{3}]{kobayashi-mtt}) that
%\[
%{\rm log}_{\omega_E}=\log_{q_E}\circ\Phi_E^{-1}:E(\bQ_p)\longrightarrow\bZ_p.
%\]
Suppose now that $p\mid N$. Let $\tilde{E}_{\rm ns}(\mathbf{F}_p)$ be the group on nonsingular points on the reduction
of $E$ modulo $p$, $E_0(K_\pp)$ be the inverse image of $\tilde{E}_{\rm ns}(\mathbf{F}_p)$
under the reduction map, and $E_1(K_\pp)$
be defined by the exactness of the sequence
\begin{equation}\label{eq:red-p}
0\longrightarrow E_1(K_\pp)\longrightarrow E_0(K_\pp)\longrightarrow\tilde{E}_{\rm ns}(\mathbf{F}_p)\longrightarrow 0.
\end{equation}
The formal group logarithm defines an injective homomorphism $\log_{\omega_E}:E(K_\pp)_{/{\rm tor}}\otimes_{}\bZ_p\rightarrow\bZ_p$
mapping $E_1(K_\pp)$ isomorphically onto $p\bZ_p$, and hence we see that
\begin{align}\label{eq:tam}
[E(K_\pp)_{/{\rm tors}}\otimes\bZ_p:\bZ_p.P]&=\frac{\#\bZ_p/(\log_{\omega_E}P)\cdot\#(E(K_\pp)/E_1(K_\pp)\otimes\bZ_p)}
{\#\bZ_p/p\bZ\cdot\#(E(K_\pp)_{\rm tors}\otimes\bZ_p)}\nonumber \\
&=[E(K_\pp):E_0(K_\pp)]_p\cdot\frac{\#\bZ_p/(\log_{\omega_E}P)\cdot\#(E_0(K_\pp)/E_1(K_\pp)\otimes\bZ_p)}
{\#\bZ_p/p\bZ_p\cdot\#(E(K_\pp)_{\rm tors}\otimes\bZ_p)},\nonumber
\end{align}
where the first equality follows from the same immediate calculation as in \cite[p.~23]{JSW},
and in the second equality $[E(K_\pp):E_0(K_\pp)]_p$ denotes the $p$-part of the
index $[E(K_\pp):E_0(K_\pp)]$. By $(\ref{eq:red-p})$, we have $E_1(K_\pp)/E_0(K_\pp)\otimes\bZ_p\simeq\tilde{E}_{\rm ns}(\mathbf{F}_p)\otimes\bZ_p$,
which is trivial by e.g. \cite[Prop.~5.1]{silverman-2} (and $p>2$).
Since clearly $E(K_\pp)_{\rm tors}\otimes\bZ_p=H^0(K_\pp,E[p^\infty])$,
we thus %obtain from $(\ref{eq:tam})$ that
conclude that
\begin{equation}\label{eq:calcul}
[E(K_\pp)_{/{\rm tors}}\otimes\bZ_p:\bZ_p.P]=[E(K_\pp):E_0(K_\pp)]_p
\cdot\frac{\#\bZ/(\frac{1}{p}\log_{\omega_E}P)}{\#H^0(K_\pp,E[p^\infty])},
\end{equation}
and substituting $(\ref{eq:calcul})$ into $(\ref{eq:321})$ we arrive at
\[
H^1_{\rm ac}(K,E[p^\infty])=\#\Sha(E/K)[p^\infty]
\cdot\left(\frac{[E(K_\pp):E_0(K_\pp)]_p\cdot\#\bZ_p/(\frac{1}{p}\log_{\omega_E}P)}
{[E(K)\otimes\bZ_p:\bZ_p.P]\cdot\#H^0(K_\pp,E[p^\infty])}\right)^2.
\]
Plugging this formula for $H^1_{\rm ac}(K,E[p^\infty])$ into \cite[Thm.~3.3.1]{JSW}
yields the equality
\begin{equation}\label{eq:control-p}
\begin{split}
\#\bZ_p/f_{\rm ac}^\Sigma(0)&=\#\Sha(E/K)[p^\infty]\cdot
\Biggl(\frac{\#\bZ_p/(\frac{1}{p}\log_{\omega_E}P)}{[E(K)\otimes\bZ_p:\bZ_p.P]}\Biggr)^2\cdot[E(K_\pp):E_0(K_\pp)]_p^2\\
&\quad\times
\prod_{\substack{w\in S\smallsetminus\Sigma\\w\nmid p\;{\rm split}}}\#H_{\rm ur}^1(K_w,E[p^\infty])\cdot\prod_{w\in\Sigma}\#H^1(K_w,E[p^\infty]),
\end{split}
\end{equation}
where $S$ is any finite set of places of $K$ %not above $p$
containing $\Sigma$ and the primes above $N$. Now, if $w\mid p$, then
\begin{equation}\label{eq:tam-p}
[E(K_\pp):E_0(K_\pp)]_p=c_w^{(p)}(E/K)
\end{equation}
by definition, while if $w\nmid p$, then
\begin{equation}\label{eq:tam-not-p}
\#H_{\rm ur}^1(K_w,E[p^\infty])=c^{(p)}_w(E/K)
\end{equation}
by \cite[Lem.~9.1]{skinner-zhang}. Since $c^{(p)}_w(E/K)=1$ unless $w\mid N$,
substituting $(\ref{eq:tam-p})$ and $ (\ref{eq:tam-not-p})$ into $(\ref{eq:control-p})$,
the proof of Theorem~\ref{thm:control-greenberg} follows.
\end{proof}

\subsubsection*{Control Theorem for Greenberg Selmer groups}

Let $\Lambda_W=\bZ_p[[W]]$ be a one-variable power series ring.
Let $M$ be an integer prime to $p$,
let $\chi$ be a Dirichlet character modulo $pM$,
and let $\mathbf{f}=\sum_{n=1}^\infty\mathbf{a}_nq^n\in\cR[[q]]$ be an ordinary $\cR$-adic cusp eigenform
of tame level $M$ and character $\chi$ (as defined in \cite[\S{3.3.9}]{SU})
defined over a local reduced finite integral extension $\cR/\Lambda_W$.

Let $\mathcal{X}^a_\cR$ the set of continuous $\bZ_p$-algebra homomorphisms
$\phi:\cR\rightarrow\overline{\bQ}_p$ whose composition with the structural map
$\Lambda_W\rightarrow\cR$ is given by $\phi(1+W)=(1+p)^{k_\phi-2}$
%for some $p$-power root of unity $\zeta$ and
for some integer $k_\phi\in\bZ_{\geqslant 2}$ called the \emph{weight} of $\phi$.
Then %in particular,
for all $\phi\in\mathcal{X}^a_\cR$ we have
\[
\mathbf{f}_\phi=\sum_{n=1}^\infty\phi(\mathbf{a}_n)q^n\in S_{k_\phi}(\Gamma_0(pM),\chi\omega^{2-k_\phi}),
\]
where $\omega$ is the Teichm\"uller character.
In this paper will only consider the case where $\chi$ is the trivial character,
in which case for all $\phi\in\mathcal{X}^a_\cR$ of weight $k_\phi\equiv 2\pmod{p-1}$, either
\begin{enumerate}
\item{} $\mathbf{f}_\phi$ is a newform on $\Gamma_0(pM)$;
\item{} $\mathbf{f}_\phi$ is the $p$-stabilization of a $p$-ordinary newform on $\Gamma_0(M)$.
\end{enumerate}
As is well-known, %(see e.g. \cite[Lem.~2.1.5]{howard-invmath}), 
for weights $k_\phi>2$ only case (2) is possible;
for $k_\phi=2$ both cases occur.

Let $k_\cR=\cR/\mathfrak{m}_\cR$ be the residue field of $\cR$, and assume
that the residual Galois representation
\[
\bar{\rho}_{\mathbf{f}}:G_\bQ:={\rm Gal}(\overline{\bQ}/\bQ)\longrightarrow{\rm GL}_2(\kappa_\cR)
\]
attached to $\mathbf{f}$ is irreducible. Then there exists a free $\cR$-module $T_{\mathbf{f}}$
of rank two equipped with a continuous $\cR$-linear action of $G_{\bQ}$ such that,
for all $\phi\in\mathcal{X}_\cR^{a}$,
%if $\cO_L$ is the ring of integers of a finite extension of $\bQ_p$ containing $\phi(\cR)$,
there is a canonical $G_\bQ$-isomorphism
\[
T_{\mathbf{f}}\otimes_{\cR}\phi(\cR)\simeq T_{\mathbf{f}_\phi},
\]
where $T_{\mathbf{f}_\phi}$ is a $G_\bQ$-stable lattice in the
Galois representation $V_{\mathbf{f}_\phi}$ associated with $\mathbf{f}_\phi$.
(Here, $T_{\mathbf{f}}$ corresponds to the Galois representation denoted $M(\mathbf{f})^*$ in \cite[Def.~7.2.5]{KLZ0};
in particular, ${\rm det}(V_{\mathbf{f}_\phi})=\epsilon^{k_\phi-1}$,
where $\epsilon:G_\bQ\rightarrow\bZ_p^\times$ is the $p$-adic cyclotomic character.)
%(see e.g. \cite[Thm.~1.4.3]{Ohta}.

Let $\Lambda_{\cR}:=\cR[[\Gamma]]$ be the anticylotomic Iwasawa algebra over $\cR$, and
consider the $\Lambda_{\cR}$-module
\[
M_{\mathbf{f}}:=T_{\mathbf{f}}\otimes_{\cR}\Lambda_{\cR}^*,
\]
where $\Lambda_\cR^*={\rm Hom}_{\rm cont}(\Lambda_{\cR},\bQ_p/\bZ_p)$ is the Pontrjagin dual
of $\Lambda_\cR$. This is equipped with a natural $G_K$-action defined similarly as for the
$\Lambda$-module $M=T\otimes_{\bZ_p}\Lambda^*$ introduced in $\S\ref{sec:defs}$.

\begin{defn}
The \emph{Greenberg Selmer group} of $E$ over $K_\infty/K$ is
\[
\mathfrak{Sel}_{\rm Gr}(K_\infty,E[p^\infty]):={\rm ker}\biggl\{H^1(K,\Ac)\longrightarrow
H^1(I_\pp,\Ac)\oplus\prod_{w\nmid p}H^1(I_w,\Ac)\biggr\}.
\]
The \emph{Greenberg Selmer group} $\mathfrak{Sel}_{\rm Gr}(K_\infty,A_{\mathbf{f}})$
for $\mathbf{f}$ over $K_\infty/K$, where $A_{\mathbf{f}}:=T_{\mathbf{f}}\otimes_{\cR}\cR^*$,
is defined by replacing $M$ by $M_{\mathbf{f}}$ in the above definition.
\end{defn}

Similarly as for the anticyclotomic Selmer groups in $\S\ref{sec:defs}$,
for any given finite set $\Sigma$ of places $w\nmid p$ of $K$,
we define $\Sigma$-imprimitive Selmer groups $\mathfrak{Sel}_{\rm Gr}^\Sigma(K_\infty,E[p^\infty])$
and $\mathfrak{Sel}_{\rm Gr}^\Sigma(K_\infty,A_{\mathbf{f}})$
by dropping the summands $H^1(I_w,\Ac)$ and $H^1(I_w,\Ac_{\mathbf{f}})$, respectively,
for the places $w\in\Sigma$ in the above definition. Let
\[
X_{\rm Gr}^\Sigma(E[p^\infty]):={\rm Hom}_{\rm cont}(\mathfrak{Sel}^\Sigma_{\rm Gr}(K_\infty,E[p^\infty]),\bQ_p/\bZ_p)
\]
be the Pontrjagin dual of $\mathfrak{Sel}^\Sigma_{\rm Gr}(K_\infty,E[p^\infty])$,
and define $X_{\rm Gr}^\Sigma(A_{\mathbf{f}})$ similarly.
\sk

We will have use for the following comparison between the Selmer groups
$\mathfrak{Sel}_{\rm Gr}(K_\infty,E[p^\infty])$ and ${\rm Sel}_\pp(K_\infty,E[p^\infty])$. Note that directly
from the definition we have an exact sequence
\begin{equation}\label{eq:defs}
0\longrightarrow{\rm Sel}_\pp(K_\infty,E[p^\infty])\longrightarrow\mathfrak{Sel}_{\rm Gr}(K_\infty,E[p^\infty])
\longrightarrow\mathcal{H}_\pp^{\rm ur}\oplus\prod_{w\nmid p}\mathcal{H}_w^{\rm ur},
\end{equation}
where $\mathcal{H}_v^{\rm ur}={\rm ker}\{H^1(K_v,M)\rightarrow H^1(I_v,M)\}$ is the set of unramified cocycles.

For a torsion $\Lambda$-module $X$, let $\lambda(X)$ (resp. $\mu(X)$) denote the
$\lambda$-invariant (resp. $\mu$-invariant) of a generator of $Ch_\Lambda(X)$.

\begin{prop}\label{prop:comparison}
Assume that $X_{\rm Gr}^\Sigma(E[p^\infty])$ is $\Lambda$-torsion. Then $X_{\rm ac}^\Sigma(E[p^\infty])$
is $\Lambda$-torsion, and we have the relations
\[
\lambda(X_{\rm Gr}^\Sigma(E[p^\infty]))=\lambda(X_{\rm ac}^\Sigma(E[p^\infty])),
\]
and
\[
\mu(X_{\rm Gr}^\Sigma(E[p^\infty]))=\mu(X_{\rm ac}^\Sigma(E[p^\infty]))+
\sum_{w\;{\rm nonsplit}}{\rm ord}_p(c_w(E/K)).
\]
\end{prop}

\begin{proof}
Since $X_{\rm ac}^\Sigma(E[p^\infty])$ is a quotient of $X_{\rm Gr}^\Sigma(E[p^\infty])$,
the first claim of the proposition is clear. Also, note that $X_{\rm Gr}^\Sigma(E[p^\infty])$ is $\Lambda$-torsion
for some $\Sigma$ if and only if it is $\Lambda$-torsion
for any finite set of primes $\Sigma$. Therefore to establish the claimed relations between Iwasawa invariants,
it suffices to consider primitive Selmer groups, i.e. $\Sigma=\emptyset$.

For primes $v\nmid p$ which are split in $K$, it is easy to see that
the restriction map $H^1(K_v,M)\rightarrow H^1(I_v,M)$
is injective (see \cite[Rem.~3.1]{pollack-weston}), and so $\mathcal{H}_v^{\rm ur}$ vanishes. %(note that this includes $v=\pp$).
Since $M^{I_\pp}=\{0\}$, the term $\mathcal{H}_\pp^{\rm ur}$ also vanishes, and
the exact sequence $(\ref{eq:defs})$ thus reduces to
\begin{equation}\label{eq:defs-}
0\longrightarrow{\rm Sel}_\pp(K_\infty,E[p^\infty])\longrightarrow\mathfrak{Sel}_{\rm Gr}(K_\infty,E[p^\infty])
\longrightarrow\prod_{\substack{w\;{\rm nonsplit}}}\mathcal{H}_w^{\rm ur}.
\end{equation}
Now, a straightforward modification of the argument in \cite[Lem.~3.4]{pollack-weston} shows
that
\[
\mathcal{H}_w^{\rm ur}\simeq(\bZ_p/p^{t_E(w)}\bZ_p)\otimes\Lambda^*,
\]
where $t_E(w):={\rm ord}_p(c^{}_w(E/K))$ is the $p$-exponent of the Tamagawa number of
$E$ at $w$, and $\Lambda^*$ is the Pontrjagin dual of $\Lambda$. In particular, $\mathcal{H}_w^{\rm ur}$
is $\Lambda$-torsion, with $\lambda(\mathcal{H}_w^{\rm ur})=0$ and $\mu(\mathcal{H}_w^{\rm ur})={\rm ord}_p(c_w(E/K))$.
Since the rightmost arrow in $(\ref{eq:defs-})$ is surjective by \cite[Prop.~A.2]{pollack-weston}, %(\emph{cf.} [\emph{loc.cit.}, Prop.~5.1]),
taking characteristic ideals in $(\ref{eq:defs-})$ the result follows.
\end{proof}

For the rest of this section, assume that %the elliptic curve
$E$ has ordinary reduction at $p$, so
that the associated newform $f\in S_2(\Gamma_0(N))$
is $p$-ordinary. Let $\mathbf{f}\in\cR[[q]]$
be the Hida family associated with $f$, let $\wp\subseteq\cR$ be the kernel
of the arithmetic map $\phi\in\mathcal{X}_{\cR}^a$ such that $\mathbf{f}_\phi$
is either $f$ itself (if $p\mid N$) or
the ordinary $p$-stabilization of $f$ (if $p\nmid N$), and set $\widetilde{\wp}:=\wp\Lambda_\cR\subseteq\Lambda_\cR$.
Since we assume that $\bar{\rho}_{E,p}$
is irreducible, so is $\bar{\rho}_{\mathbf{f}}$.

\begin{thm}\label{thm:control-hida}
Let $S_p$ be the places of $K$ above $p$, and assume that
$\Sigma\cup S_p$ contains all places of $K$ at which $T$ is ramified.
Then %the natural map $(\ref{eq:can})$ induces a
there is a canonical isomorphism
\[
X_{\rm Gr}^\Sigma(E[p^\infty])\simeq X_{\rm Gr}^\Sigma(A_{\mathbf{f}})/\widetilde{\wp} X_{\rm Gr}^\Sigma(A_{\mathbf{f}}).
\]
\end{thm}

\begin{proof}
%This is standard. For the convenience of the reader, we give a proof following some of
This follows from a slight variation of the arguments proving \cite[Prop.~3.7]{SU}
(see also \cite[Prop.~5.1]{Ochiai-MC}). Since
$M\simeq M_{\mathbf{f}}[\widetilde{\wp}]$,
by Pontrjagin duality it suffices to show that the canonical map
\begin{equation}\label{eq:can}
{\rm Sel}_{\rm Gr}^\Sigma(K_\infty,M_{\mathbf{f}}[\widetilde{\wp}])
\longrightarrow{\rm Sel}_{\rm Gr}^\Sigma(K_\infty,M_{\mathbf{f}})[\widetilde{\wp}]
\end{equation}
%induced by the canonical isomorphism
%$\Lambda_{\cR}^*[\widetilde{\wp}]\simeq(\Lambda_{\cR}/\widetilde{\wp})^*$
is an isomorphism.
Note that our assumption on $S:=\Sigma\cup S_p$ implies that
\begin{equation}\label{eq:useful}
{\rm Sel}_\pp^\Sigma(K_\infty,M_{?})={\rm ker}\biggl\{H^1(G_{K,S},M_?)\overset{{\rm loc}_\pp}
\longrightarrow\frac{H^1(K_\pp,M_?)}{H^1_{\rm Gr}(K_\pp,M_?)}\biggr\},
\end{equation}
where $M_?=M_{\mathbf{f}}[\widetilde{\wp}]$ or $M_{\mathbf{f}}$,
$G_{K,S}$ is the Galois group of the maximal extension of $K$ unramified outside $S$, and
\[
H^1_{\rm Gr}(K_\pp,M_{?}):={\rm ker}
\bigl\{H^1(K_\pp,M_?)\longrightarrow H^1(I_\pp,M_?)\bigr\}.
\]

As shown in the proof of \cite[Prop.~3.7]{SU}
(taking $A=\Lambda_{\cR}$ and $\mathfrak{a}=\widetilde{\wp}$ in \emph{loc.cit.}),
we have $H^1(G_{K,S},M_{\mathbf{f}}[\widetilde{\wp}])=H^1(G_{K,S},M_{\mathbf{f}})[\widetilde{\wp}]$.
%(note that this part of their argument does not require any hypothesis
%on the action of inertia at $p$ on $T$).
On the other hand, using that $G_{K_\pp}/I_\pp$ %is procyclic (and hence of
has cohomological dimension one, we immediately see that
\[
H^1(K_\pp,M_?)/H^1_{\rm Gr}(K_\pp,M_?)
\simeq H^1(I_\pp,M_?)^{G_{K_\pp}},
\]
%and similarly putting $T\otimes_AA^*$ coefficients everywhere.
From the long exact sequence in $I_\pp$-cohomology associated with
$0\rightarrow\Lambda_{\cR}^*[\widetilde{\wp}]\rightarrow\Lambda_{\cR}^*\rightarrow\widetilde{\wp}\Lambda_{\cR}^*\rightarrow 0$
tensored with $T_{\mathbf{f}}$, we obtain
\[
(M_{\mathbf{f}}^{I_\pp}/(T_{\mathbf{f}}\otimes_{\cR}\widetilde{\wp}\Lambda_{\cR}^*)^{I_\pp})^{G_{K_\pp}}\simeq{\rm ker}
\bigl\{H^1(I_\pp,M_{\mathbf{f}}[\widetilde{\wp}])^{G_{K_\pp}}\longrightarrow H^1(I_\pp,M_{\mathbf{f}})^{G_{K_\pp}}\bigr\}.
\]
Since $H^0(I_\pp,M_{\mathbf{f}})=\{0\}$, we thus have a commutative diagram
\[
\xymatrix{
H^1(G_{K,S},M_{\mathbf{f}}[\widetilde{\wp}])\ar[rr]^-{{\rm loc}_\pp}\ar[d]^{\simeq} && H^1(K_\pp,M_{\mathbf{f}}[\widetilde{\wp}])/H^1_{\rm Gr}(K_\pp,M_{\mathbf{f}}[\widetilde{\wp}])\ar[d]\\
H^1(G_{K,S},M_{\mathbf{f}})[\widetilde{\wp}]\ar[rr]^-{{\rm loc}_\pp} && H^1(K_\pp,M_{\mathbf{f}})/H^1_{\rm Gr}(K_\pp,M_{\mathbf{f}})
}
\]
in which the right vertical map is injective. In light of $(\ref{eq:useful})$, the result follows.
\end{proof}

\section{A $p$-adic Waldspurger formula}\label{sec:padicL}

Let $E$, $p$, and $K$ be an introduced in $\S\ref{sec:defs}$.
%in particular $p=\pp\overline{\pp}$ splits in $K$.
In this section, we assume in addition
that $K$ satisfies the following Heegner hypothesis relative to the
square-free integer $N$:
\begin{equation}\label{eq:Heeg-hyp}
\textrm{every prime factor of $N$ is either split or ramified in $K$.}\tag{Heeg}
\end{equation}

\subsubsection*{Anticyclotomic $p$-adic $L$-function}

Let $f=\sum_{n=1}^\infty a_nq^n\in S_2(\Gamma_0(N))$
be the newform associated with $E$.
Denote by $R_0$ the completion of the ring of integers
of the maximal unramified extension of $\bQ_p$, %in $\overline{\bQ}_p$,
and set $\Lambda_{R_0}:=\Lambda\hat{\otimes}_{\bZ_p}R_0$, where
as before $\Lambda=\bZ_p[[\Gamma]]$ is the anticyclotomic Iwasawa algebra. % as before.

\begin{thm}\label{thm:L-bdp}
%Assume that $p=\pp\overline{\pp}$ splits in $K$ and that hypothesis $({\rm Heeg})$ holds.
There exists a $p$-adic $L$-function
$L_p^{}(f)\in\Lambda_{R_0}$ such that if $\hat{\phi}:\Gamma\rightarrow\bC_p^\times$ is the $p$-adic avatar
of an unramified anticyclotomic Hecke character $\phi$
with infinity type $(-n,n)$ with $n>0$, then
\begin{align*}
L_{p}^{}(f,\hat\phi)&=
\Gamma(n)\Gamma(n+1)
\cdot(1-a_pp^{-1}\phi(\pp)+\varepsilon_p\phi^2(\pp))^2
%\cdot\psi^{-1}(\mathfrak{N}^+)\cdot c\varepsilon(f)\cdot w_K^2D_K\\
\cdot\Omega_p^{4n}\cdot\frac{L(f/K,\phi,1)}{\pi^{2n+1}\cdot\Omega_K^{4n}},
\end{align*}
where %$w_K:=\vert\cO_K^\times\vert$, and
$\varepsilon_p=p^{-1}$ if $p\nmid N$ and $\varepsilon_p=0$ otherwise, and
$\Omega_p\in\unr^\times$ and $\Omega_K\in\bC^\times$  are CM periods. %attached to $K$. (see \cite[\S{2.5}]{cas-hsieh1}).
%, and $\varepsilon(f)$ is the global root number of $f$.
%\cite[(137)]{bdp1} and \cite[(140)]{bdp1}, respectively.
\end{thm}

\begin{proof}
Let $\psi$ be an anticyclotomic Hecke character of $K$ of infinity type $(1,-1)$
and conductor prime to $p$, let $\mathscr{L}_{\pp,\psi}(f)\in\Lambda_{\unr}^{}$ be as
in \cite[Def.~3.7]{cas-hsieh1}, and set
\[
L_{p}(f):={\rm Tw}_{\psi^{-1}}(\mathscr{L}_{\pp,\psi}(f)),
\]
where ${\rm Tw}_{\psi^{-1}}:\Lambda_{\unr}^{}\rightarrow\Lambda_{\unr}^{}$
is the $\unr$-linear isomorphism given by $\gamma\mapsto\psi^{-1}(\gamma)\gamma$
for $\gamma\in\Gamma$. If $p\nmid N$, the interpolation property for $L_p(f)$
is a reformulation of \cite[Thm.~3.8]{cas-hsieh1}.
Since the construction in \cite[\S{3.3}]{cas-hsieh1}
readily extends to the case $p\mid N$, with the $p$-adic multiplier $e_\pp(f,\phi)$
in \emph{loc.cit.} reducing to $1-a_pp^{-1}\phi(\pp)$ for unramified $\phi$
(\emph{cf.} \cite[Thm.~2.10]{cas-split}), the result follows.
\end{proof}

%We may not need this here (but we certainly do in \cite{cas-wan-big}!):

%\begin{prop}
%$\mu(\mathscr{L}_\pp^{\tt BDP}(f/K))=0$.
%\end{prop}

%\begin{proof}
%\end{proof}

If $\Sigma$ is any finite set of place of $K$ not lying above $p$, we define the ``$\Sigma$-imprimitive''
$p$-adic $L$-function $L_p^\Sigma(f)$ by
\begin{equation}\label{eq:def-imp}
L_p^\Sigma(f):=L_p(f)\times\prod_{w\in\Sigma}P_w(\epsilon\Psi^{-1}(\gamma_w))\in\Lambda_{R_0},
\end{equation}
where $P_w(X):={\rm det}(1-X\cdot{\rm Frob}_w\vert V^{I_w})$,
$\epsilon:G_K\rightarrow\bZ_p^\times$ is the $p$-adic cyclotomic character,
${\rm Frob}_w\in G_K$ is a geometric Frobenius element at $w$, and
$\gamma_w$ is the image of ${\rm Frob}_w$ in $\Gamma$.

\subsubsection*{$p$-adic Waldspurger formula}
%Similarly as in \cite{JSW},

We will have use for the following %``$p$-adic Waldspurger formula''
formula for the value at the trivial character $\mathds{1}$
of the $p$-adic $L$-function of Theorem~\ref{thm:L-bdp}.

Recall that $E/\bQ$ is assumed to be semistable. From now on, we shall also
assume that $E$ is an optimal quotient of the new part of $J_0(N)={\rm Jac}(X_0(N))$
in the sense of \cite[\S{2}]{mazur-isogenies}, and fix a corresponding modular parametrization
\[
\pi:X_0(N)\longrightarrow E
\]
sending the cusp $\infty$ to the origin of $E$. If
$\omega_E$ a N\'eron differential on $E$, and $\omega_f=\sum a_nq^n\frac{dq}{q}$
is the one-form on $J_0(N)$ associated with $f$, then
\begin{equation}\label{eq:manin}
\pi^*(\omega_E)=c\cdot\omega_f,
\end{equation}
for some $c\in\bZ_{(p)}^\times$ (see \cite[Cor.~4.1]{mazur-isogenies}).

\begin{thm}\label{thm:bdp}
The following equality holds up to a $p$-adic unit:
\[
L_p(f,\mathds{1})=%\bigl((1-a_pp^{-1}+\varepsilon_p)\log_{\omega_E}P_K\bigr)^2,
(1-a_pp^{-1}+\varepsilon_p)^2\cdot(\log_{\omega_E}P_K)^2,
\]
where $\varepsilon_p=p^{-1}$ if $p\nmid N$ and $\varepsilon_p=0$ otherwise, and
$P_K\in E(K)$ is a Heegner point.
\end{thm}

\begin{proof}
This follows from \cite[Thm.~5.13]{bdp1} and \cite[Thm.~4.9]{cas-hsieh1} in the case $p\nmid N$
and \cite[Thm.~2.11]{cas-split} in the case $p\mid N$. Indeed, in our case,
the generalized Heegner cycles $\Delta$ constructed in either of these references
are of the form
\[
\Delta=[(A,A[\mathfrak{N}])-(\infty)]\in J_0(N)(H),
\]
where $H$ is the Hilbert class field of $K$, and $(A,A[\mathfrak{N}])$
is a CM elliptic curve equipped with a cyclic $N$-isogeny.
Letting $F$ denote the $p$-adic completion of $H$, the aforementioned
references then yield the equality
\begin{equation}\label{eq:BDP-formula}
L_p(f,\mathds{1})=
(1-a_pp^{-1}+\varepsilon_p)^2
\cdot\Biggl(\sum_{\sigma\in{\rm Gal}(H/K)}{\rm AJ}_{F}(\Delta^\sigma)(\omega_f)\Biggr)^2.
\end{equation}
By \cite[Ex.~3.10.1]{BK}, the $p$-adic Abel--Jacobi map appearing in $(\ref{eq:BDP-formula})$
is related to the formal group logarithm on $J_0(N)$ by the formula
\[
{\rm AJ}_{F}(\Delta)(\omega_f)=\log_{\omega_f}(\Delta),
\]
and by $(\ref{eq:manin})$ we have the equalities up to a $p$-adic unit:
\[
\log_{\omega_f}(\Delta)=\log_{\pi^*(\omega_E)}(\pi(\Delta))=\log_{\omega_E}(\pi(\Delta))
\]
Thus, taking $P_K:=\sum_{\sigma\in{\rm Gal}(H/K)}\pi(\Delta^\sigma)\in E(K)$, the result follows.
\end{proof}

\section{Main Conjectures}\label{sec:MC}

Let $\mathbf{f}\in\cR[[q]]$ be an ordinary $\cR$-adic cusp eigenform of tame level $M$
as in Section~\ref{sec:Selmer} (so $p\nmid M$),
with associated residual representation $\bar{\rho}_{\mathbf{f}}$.
Letting $D_p\subseteq G_{\bQ}$ be a fixed decomposition group at $p$,
 we say that $\bar{\rho}_{\mathbf{f}}$ is \emph{$p$-distinguished}
if the semisimplication of $\bar{\rho}_{\mathbf{f}}\vert_{D_p}$ is the direct sum of
two distinct characters.

%Since it will suffice for our purposes in this paper, we
%shall assume that $\bar{\rho}_{\mathbf{f}}$ takes values in ${\rm GL}_2(\mathbf{F}_p)$.
Let $K$ be an imaginary quadratic field
in which $p=\pp\overline{\pp}$ splits,
and which satisfies hypothesis (\ref{eq:Heeg-hyp})
from Section~\ref{sec:padicL} relative to $M$.
%which we shall assume to be square-free.

For the next statement, note that for any eigenform $f$ defined over a finite extension $L/\bQ_p$
with associated Galois representation $V_f$, we may define the Selmer group %$X_{\rm ac}^\Sigma(A_f)$ and
$X_{\rm Gr}^\Sigma(A_f)$ as in $\S\ref{sec:control}$, replacing $T=T_pE$ by a fixed $G_{\bQ}$-stable
$\cO_L$-lattice in $V_f$, and setting $A_f:=V_f/T_f$.

%We begin by recalling a slight variation of the following result from \cite{cas-BF}.

\begin{thm}\label{thm:cas-BF}
Let $f\in S_2(\Gamma_0(M))$ be a $p$-ordinary newform %weight $2$ and
of level $M$, with $p\nmid M$,
and let $\bar\rho_f$ be the associated residual representation. Assume that:
\begin{itemize}
\item{} $M$ is square-free;
%\item{} $\bar{\rho}_{f}$ is $p$-distinguished,
\item{} $\bar{\rho}_f$ is ramified at every prime $q\mid M$ which is nonsplit in $K$,
and there is at least one such prime;
\item{} $\bar{\rho}_{f}\vert_{G_K}$ is irreducible.
%\item{} $\bar{\rho}_{\mathbf{f}}$ is ramified at some prime $q\mid N$ nonsplit in $K$.
\end{itemize}
If $\Sigma$ is any finite set of prime not lying above $p$, then
$X_{\rm Gr}^\Sigma(A_{f})$
is $\Lambda$-torsion, and
\[
Ch_{\Lambda}(X^\Sigma_{\rm Gr}(A_{f}))\Lambda_{R_0}=
(L_p^\Sigma(f)),
\]
where $L_p^{\Sigma}(f)$ is as in $(\ref{eq:def-imp})$.
%Theorem~\ref{thm:L-bdp}.
\end{thm}

\begin{proof}
As in the proof of \cite[Thm.~6.1.6]{JSW},
the result for an arbitrary finite set $\Sigma$ follows immediately from
the case $\Sigma=\emptyset$, which is the content of \cite[Thm.~3.4]{cas-BF}.
(In \cite{cas-BF} it is assumed that $f$
has rational Fourier coefficients
%, i.e. $f$ corresponds to an elliptic curve $E_f/\bQ$, and that $(M,D_K)=1$,
but the extension of the aforementioned
result %from \emph{loc.cit.}
to the setting considered here is immediate.)
\end{proof}

Recall that $\Lambda_\cR$ denotes the anticyclotomic Iwasawa algebra over $\cR$,
and set $\Lambda_{\cR,R_0}:=\Lambda_{\cR}\hat{\otimes}_{\bZ_p}R_0$. For any
$\phi\in\mathcal{X}_{\cR}^a$, set $\widetilde{\wp}_\phi:={\rm ker}(\phi)\Lambda_{\cR,R_0}$.

\begin{thm}\label{thm:2-var-IMC}
Let $\Sigma$ be a finite set of places of $K$ not above $p$. Letting $M$ be the tame
level of $\mathbf{f}$, assume that:
\begin{itemize}
\item{} $M$ is square-free;
\item{} $\bar{\rho}_{\mathbf{f}}$ is ramified at every prime $q\mid M$ which is nonsplit in $K$,
and there is at least one such prime;
\item{} $\bar{\rho}_{\mathbf{f}}\vert_{G_K}$ is irreducible;
\item{} $\bar{\rho}_{\mathbf{f}}$ is $p$-distinguished.
\end{itemize}
Then $X_{\rm Gr}^\Sigma(A_{\mathbf{f}})$
is $\Lambda_{\cR}$-torsion, and
\[
Ch_{\Lambda_\cR}(X_{\rm Gr}^\Sigma(A_{\mathbf{f}}))\Lambda_{\cR,R_0}=
(L_p^\Sigma(\mathbf{f})),
\]
where $L_p^\Sigma(\mathbf{f})\in\Lambda_{\cR,R_0}$ is such that
\begin{equation}\label{eq:sp-property}
L_p^\Sigma(\mathbf{f})\;{\rm mod}\;\widetilde{\wp}_\phi=L_p^{\Sigma}(\mathbf{f}_\phi)
\end{equation}
for all $\phi\in\mathcal{X}_{\cR}^a$.
%for all $\phi\in\mathcal{X}_{\cR}^a$.
%with $L_p^{\Sigma}(\mathbf{f}_\phi)\in\Lambda_{R_0}$ as in Theorem~\ref{thm:L-bdp}.
\end{thm}

\begin{proof}
Let $\mathscr{L}_{\pp,\boldsymbol{\xi}}(\mathbf{f})\in\Lambda_{\cR,R_0}$ be the two-variable
anticyclotomic $p$-adic $L$-function constructed in \cite[\S{2.6}]{cas-2var}, and set
\[
L_p(\mathbf{f}):={\rm Tw}_{\boldsymbol{\xi}^{-1}}(\mathscr{L}_{\pp,\boldsymbol{\xi}}(\mathbf{f})),
\]
where $\boldsymbol{\xi}$ is the $\cR$-adic character constructed in \emph{loc.cit.} from a Hecke character $\lambda$
of infinity type $(1,0)$ and conductor prime to $p$, and ${\rm Tw}_{\boldsymbol{\xi}^{-1}}:\Lambda_{\cR,R_0}\rightarrow\Lambda_{\cR,R_0}$
is the $R_0$-linear isomorphism given by $\gamma\mapsto\boldsymbol{\xi}^{-1}(\gamma)\gamma$ for $\gamma\in\Gamma$.
Viewing $\lambda$ as a character on $\mathbb{A}_K^\times$, let $\lambda^\tau$ denote
the composition of $\lambda$ with the action of complex conjugation on $\mathbb{A}_K^\times$.
If the character $\psi$ appearing in the proof of Theorem~\ref{thm:L-bdp} is taken to be
$\lambda^{1-\tau}:=\lambda/\lambda^\tau$,
then the proof of \cite[Thm.~2.11]{cas-2var} shows that
$L_p(\mathbf{f})$ reduces to $L_p^{}(\mathbf{f}_\phi)$ modulo $\widetilde{\wp}_\phi$
%\[
%L_p(\mathbf{f})\;{\rm mod}\;\widetilde{\wp}_\phi= L_p^{}(\mathbf{f}_\phi)
%\]
for all $\phi\in\mathcal{X}_{\cR}^a$. Similarly as in $(\ref{eq:def-imp})$,
if for any $\Sigma$ as above we set
\[
L_p^\Sigma(\mathbf{f}):=L_p(\mathbf{f})\times\prod_{w\in\Sigma}P_{\mathbf{f},w}(\epsilon\Psi^{-1}(\gamma_w))\in\Lambda_{\cR,R_0},\nonumber
\]
where $P_{\mathbf{f},w}(X):={\rm det}(1-X\cdot{\rm Frob}_w\vert(T_\mathbf{f}\otimes_{\cR}F_\cR)^{I_w})$,
with $F_\cR$ the fraction field of $\cR$, the specialization property $(\ref{eq:sp-property})$ thus follows.

Let $\phi\in\mathcal{X}_{\cR}^a$ be such that $\mathbf{f}_\phi$ is the $p$-stabilization of
a $p$-ordinary newform $f\in S_2(\Gamma_0(M))$. %of weight $2$.   %and let $X_{\rm Gr}^\Sigma(A_f)$ be the associated Greenberg Selmer group.
By Theorem~\ref{thm:2-var-IMC}, the associated %Greenberg Selmer group
$X_{\rm Gr}^\Sigma(A_f)$ is $\Lambda$-torsion, and we have
\begin{equation}\label{eq:IMC-wt2}
Ch_{\Lambda}(X_{\rm Gr}^\Sigma(A_{f}))\Lambda_{R_0}=
(L_p^\Sigma(f)).
\end{equation}
%where $L_p^{\Sigma}(f)$ is as in $(\ref{eq:def-imp})$.
In particular, by Theorem~\ref{thm:control-hida} (with $A_f$ in place of $E[p^\infty]$)
it follows that $X_{\rm Gr}^\Sigma(A_{\mathbf{f}})$ is $\Lambda_{\cR}$-torsion.
On the other hand, from \cite[Thm.~1.1]{wanIMC} we have the divisibility
\begin{equation}\label{eq:wan-thm}
Ch_{\Lambda_\cR}(X_{\rm Gr}^\Sigma(A_{\mathbf{f}}))\Lambda_{\cR,R_0}\subseteq
(\mathcal{L}_p^\Sigma(\mathbf{f})^-)
\end{equation}
in $\Lambda_{\cR,R_0}$, where $\mathcal{L}_p^\Sigma(\mathbf{f})^-$ is the projection onto
$\Lambda_{\cR,R_0}$ of the %three-variable
$p$-adic $L$-function constructed in \cite[\S{7.4}]{wanIMC}.
Since a straightforward extension of the calculations in \cite[\S{5.3}]{JSW} shows that
\begin{equation}\label{eq:compare}
(\mathcal{L}_p^\Sigma(\mathbf{f})^-)=(L_p^\Sigma(\mathbf{f}))
\end{equation}
as ideals in $\Lambda_{\cR,R_0}$, the result follows from an application of \cite[Lem.~3.2]{SU} using
$(\ref{eq:IMC-wt2})$, $(\ref{eq:wan-thm})$, and $(\ref{eq:compare})$. (Note that the possible powers of $p$
in \cite[Cor.~5.3.1]{JSW} only arise when there are primes $q\mid M$ inert in $K$, but these are excluded
by our hypothesis (Heeg) relative to $M$.)
\end{proof}

In order to deduce from Theorem~\ref{thm:2-var-IMC} the anticyclotomic main conjecture
for arithmetic specializations of $\mathbf{f}$ (especially in the cases where
the conductor of $\mathbf{f}_\phi$ is divisible by $p$, which are not covered by Theorem~\ref{thm:cas-BF}), 
we will require the following technical result.

\begin{lem}\label{lem:ps}
Let $X_{\rm Gr}^\Sigma(A_{\mathbf{f}})_{\rm null}$ be the largest pseudo-null $\Lambda_\cR$-submodule
of $X_{\rm Gr}^\Sigma(A_{\mathbf{f}})$, let $\wp\subseteq\cR$ be a height one prime, and let
$\widetilde{\wp}:=\wp\Lambda_\cR$. % be such that $\widetilde{\wp}\cap\cR=\wp$.
With hypotheses as in Theorem~\ref{thm:2-var-IMC}, the quotient
\[
X_{\rm Gr}^\Sigma(A_{\mathbf{f}})_{\rm null}/\widetilde{\wp}X_{\rm Gr}^\Sigma(A_{\mathbf{f}})_{\rm null}
\]
is a pseudo-null $\Lambda_\cR/\widetilde{\wp}$-module.
\end{lem}

\begin{proof}
Using %once again
$(\ref{eq:useful})$ as in the proof of Theorem~\ref{thm:control-hida}
and considering the obvious commutative diagram obtained by applying the map given by multiplication by $\widetilde{\wp}$,
the proof of \cite[Lem.~7.2]{Ochiai-MC} carries through with only small changes. (Note that the argument in \emph{loc.cit.}
requires knowing that $X_\ac^\Sigma(M_{\mathbf{f}}[\widetilde{\wp}])$ is $\Lambda_\cR/\widetilde{\wp}$-torsion,
but this follows immediately from Theorem~\ref{thm:2-var-IMC} and the isomorphism of Theorem~\ref{thm:control-hida}.)
\end{proof}

%Combined with the control result of Theorem~\ref{thm:control-hida}:

For the next result, let $E/\bQ$ be an elliptic curve of square-free conductor $N$, and
assume that $K$ satisfies hypothesis (\ref{eq:Heeg-hyp}) relative to $N$, 
and that $p=\pp\overline{\pp}$ splits in $K$.

%As usual, we let $\bar{\rho}_{E,p}:G_{\bQ}\rightarrow{\rm GL}_2(\mathbf{F}_p)$
%be the mod $p$ Galois representation associated with $E$.

\begin{thm}\label{thm:IMC}
Assume that
$\bar{\rho}_{E,p}:G_{\bQ}\rightarrow{\rm Aut}_{\mathbf{F}_p}(E[p])\simeq{\rm GL}_2(\mathbf{F}_p)$
is irreducible and ramified at every prime $q\mid N$ which is nonsplit in $K$, and assume
that there is at least one such prime. Then
$Ch_{\Lambda}(X_{\ac}(E[p^\infty]))$ is $\Lambda$-torsion and
\[
Ch_{\Lambda}(X_{\ac}(E[p^\infty]))\Lambda_{R_0}=(L_p(f)).
\]
\end{thm}

\begin{proof}
If $E$ has good ordinary (resp. supersingular)
supersingular reduction at $p$, the result  
follows from \cite[Thm.~3.4]{cas-BF} (resp.
\cite[Thm.~5.1]{cas-wan-ss}). 
(Note that bBy \cite[Lem.~2.8.1]{skinner} the hypotheses in Theorem~\ref{thm:IMC} 
imply that $\bar{\rho}_{E,p}\vert_{G_K}$ is irreducible.) 
Since the conductor of $N$ is square-free, it remains to consider the case in which
$E$ has multiplicative reduction at $p$. The associated newform
$f=\sum_{n=1}^\infty a_nq^n\in S_2(\Gamma_0(N))$ %associated with $E$
then satisfies $a_p=\pm{1}$ (see e.g. \cite[Lem.~2.1.2]{skinner-mult});
in particular, $f$ is $p$-ordinary. Let $\mathbf{f}\in\cR[[q]]$ be the ordinary $\cR$-adic cusp eigenform of tame level
$N_0:=N/p$ attached to $f$, so that $\mathbf{f}_\phi=f$ for some $\phi\in\mathcal{X}_\cR^a$.
Let $\wp:={\rm ker}(\phi)\subseteq\cR$ be the associated height one prime, and set
\[
\widetilde{\wp}:=\wp\Lambda_{\cR,R_0},\quad
\Lambda_{\wp,R_0}:=\Lambda_{\cR,R_0}/\widetilde{\wp},\quad
\widetilde{\wp}_0:=\widetilde{\wp}\cap\Lambda_\cR,\quad
\Lambda_{\wp}:=\Lambda_\cR/\wp_0.
\]

Let $\Sigma$ be a finite set of places of $K$ not dividing $p$ containing
the primes above $N_0D$, where $D$ is the discriminant of $K$.
As shown in the proof of \cite[Thm.~6.1.6]{JSW}, it suffices to show that
\begin{equation}\label{eq:impr}
Ch_{\Lambda}(X_{\ac}^\Sigma(E[p^\infty]))\Lambda_{R_0}=(L_p^\Sigma(f)).
\end{equation}

Since $\mathbf{f}$ specializes %to the ordinary $p$-stabilization of
$f$, which has weight 2 and trivial nebentypus, the residual representation $\bar{\rho}_{\mathbf{f}}\simeq\bar{\rho}_{E,p}$
is automatically $p$-distinguished (see \cite[Rem.~7.2.7]{KLZ2}). Thus our assumptions %on $\overline{\rho}_{E,p}$
imply that the hypotheses in Theorem~\ref{thm:2-var-IMC} are satisfied, which
combined with Theorem~\ref{thm:control-hida} show that $X_{\rm Gr}^\Sigma(E[p^\infty])$
is $\Lambda$-torsion. Moreover, letting $\mathfrak{l}$ be any height one prime of $\Lambda_{\wp,R_0}$
and setting $\mathfrak{l}_0:=\mathfrak{l}\cap\Lambda_\wp$, by Theorem~\ref{thm:control-hida} we have
\begin{equation}\label{eq:control}
{\rm length}_{(\Lambda_\wp)_{\frakl_0}}(X_{\rm Gr}^\Sigma(E[p^\infty])_{\frakl_0})=
{\rm length}_{(\Lambda_\wp)_{\frakl_0}}((X^\Sigma_{\rm Gr}(A_{\mathbf{f}})/\widetilde{\wp}_0 X^\Sigma_{\rm Gr}(A_{\mathbf{f}}))_{\frakl_0}).
\end{equation}
On the other hand, if $\widetilde{\mathfrak{l}}\subseteq\Lambda_{\cR,R_0}$ maps to $\mathfrak{l}$
under the specialization map $\Lambda_{\cR,R_0}\rightarrow\Lambda_{\wp,R_0}$ and we set
$\widetilde{\mathfrak{l}}_0:=\widetilde{\mathfrak{l}}\cap\Lambda_\cR$, by Theorem~\ref{thm:2-var-IMC} we have
\begin{equation}\label{eq:2-var-IMC}
{\rm length}_{(\Lambda_\cR)_{\widetilde{\frakl}_0}}(X_{\rm Gr}^\Sigma(A_{\mathbf{f}})_{\widetilde{\frakl}_0})
={\rm ord}_{\widetilde{\frakl}}(L_p^\Sigma(\mathbf{f})\;{\rm mod}\;\widetilde{\wp})
={\rm ord}_\frakl(L_p^\Sigma(f)).
\end{equation}
Since Lemma~\ref{lem:ps} implies the equality
\[
{\rm length}_{(\Lambda_\wp)_{\frakl_0}}((X^\Sigma_{\rm Gr}(A_{\mathbf{f}})/\widetilde{\wp}_0 X^\Sigma_{\rm Gr}(A_{\mathbf{f}}))_{\frakl_0})
={\rm length}_{(\Lambda_{\cR})_{\widetilde{\frakl}_0}}(X_{\rm Gr}^\Sigma(A_{\mathbf{f}})_{\widetilde{\frakl}_0}),
\]
combining $(\ref{eq:control})$ and $(\ref{eq:2-var-IMC})$ we conclude that
\[
{\rm length}_{(\Lambda_\wp)_{\frakl_0}}(X_{\rm Gr}^\Sigma(E[p^\infty])_{\frakl_0})
={\rm ord}_\frakl(L_p^\Sigma(f))
\]
for every height one prime $\frakl$ of $\Lambda_{\wp,R_0}$, and so
\begin{equation}\label{eq:imprim}
Ch_{\Lambda}(X_{\rm Gr}^\Sigma(E[p^\infty]))\Lambda_{R_0}=(L_p^\Sigma(f)).
\end{equation}
%as ideals in $\Lambda_{R_0}$.
Finally, since our hypothesis on $\bar{\rho}_{E,p}$ implies that $c_w(E/K)$ is a $p$-adic
unit for every prime $w$ nonsplit in $K$ (see e.g. \cite[Def.~3.3]{pollack-weston}),
we have $Ch_{\Lambda}(X_{\rm Gr}^\Sigma(E[p^\infty]))=Ch_{\Lambda}(X_{\rm ac}^\Sigma(E[p^\infty]))$
by Proposition~\ref{prop:comparison}.
Equality $(\ref{eq:imprim})$ thus reduces to $(\ref{eq:impr})$,
and the proof of Theorem~\ref{thm:IMC} follows.
\end{proof}

\section{Proof of Theorem~A}

Let $E/\bQ$ be a semistable elliptic curve of conductor $N$ as in the statement of Theorem~A;
in particular, we note that there exists a prime $q\neq p$ such that $E[p]$ is ramified at $q$.
Indeed, if $p\mid N$ this follows by hypothesis, %(see Lemma~\ref{lem:tam} below),
while if $p\nmid N$ the existence
of such $q$ follows from Ribet's level lowering theorem \cite[Thm~1.1]{ribet-eps}, as explained in the first
paragraph of \cite[\S{7.4}]{JSW}.

%We record the following well-known result for our later use.

%\begin{lem}\label{lem:tam}
%Let $\ell$ be a prime factor of $N$, and let $\Delta_\ell$ denote the minimal discriminant
%of $E$ at $\ell$. For any prime $p$, if %$p\nmid{\rm ord}_\ell(\Delta_\ell)$
%$E[p]$ is ramified at $\ell$, then ${\rm ord}_p(c_w(E/K))=0$ for all $w\mid\ell$.
%\end{lem}

%\begin{proof}
%Suppose first that $p\neq\ell$. Then by \cite[Ch.~V, Prop.~6.1]{silverman-2} the condition $p\nmid{\rm ord}_\ell(\Delta_\ell)$
%is equivalent to $E[p]$ being ramified at $\ell$, and by \cite[Lemma~6.3]{zhang-Kolyvagin} (and the discussion right after it)
%the latter condition implies that ${\rm ord}_p(c_w(E/K))=0$ for all $w\mid\ell$.
%On the other hand, if $p=\ell$ then by \cite[Prop.~5]{serre-DMJ} the condition $p\nmid{\rm ord}_p(\Delta_p)$
%is equivalent to $E[p]$ not being `finite at $p$', and by \cite[Lemma~9.3]{skinner-zhang} this implies that
%${\rm ord}_p(c_w(E/K))=0$ for all $w\mid p$.
%\end{proof}

\begin{proof}[Proof of Theorem~A]

Choose an imaginary quadratic field $K=\bQ(\sqrt{D})$ of discriminant $D<0$ such that
\begin{itemize}
\item{} $q$ is ramified in $K$;
\item{} every prime factor $\ell\neq q$ of $N$ splits in $K$;
\item{} $p$ splits in $K$;
\item{} $L(E^{D},1)\neq 0$.
\end{itemize}
(Of course, when $p\mid N$ the third condition is redundant.)
By Theorem~\ref{thm:IMC} and Proposition~\ref{thm:bdp} we have the equalities
\begin{equation}\label{eq:IMC+BDP}
\#\bZ_p/f_{\rm ac}(0)=\#\bZ_p/L_p(f,\mathds{1})
=\#\left(\bZ_p/(1-a_pp^{-1}+\varepsilon_p)\log_{\omega_E}P_K\right)^2,
\end{equation}
where $\varepsilon_p=p^{-1}$ if $p\nmid N$ and $\varepsilon_p=0$ otherwise, and $P_K\in E(K)$
is a Heegner point. Since we assume that ${\rm ord}_{s=1}L(E,s)=1$, our last hypothesis on $K$
implies that ${\rm ord}_{s=1}L(E/K,s)=1$, and so %we have
$P_K$ has infinite order, ${\rm rank}_{\bZ}(E(K))=1$ and $\#\Sha(E/K)<\infty$ by the work of
Gross--Zagier and Kolyvagin. This verifies the hypotheses
in Theorem~\ref{thm:control-greenberg}, which (taking $\Sigma=\emptyset$ and $P=P_K$)
yields a formula for $\#\bZ_p/f_{\rm ac}(0)$ that combined with $(\ref{eq:IMC+BDP})$
immediately leads to
\begin{equation}\label{eq:+control}
{\rm ord}_p(\#\Sha(E/K)[p^\infty])=2\cdot{\rm ord}_p([E(K):\bZ.{P}_K])
-\sum_{\substack{w\mid N^+}}{\rm ord}_p(c_w(E/K)),
\end{equation}
where $N^+$ is the product of the prime factors of $N$ which are split in $K$.
Since $E[p]$ is ramified at $q$,
%by Lemma~\ref{lem:tam}
we have ${\rm ord}_p(c_w(E/K))=0$ for every prime $w\mid q$
(see e.g. \cite[Lem.~6.3]{zhang-Kolyvagin} and the discussion right after it),
and since $N^+=N/q$ by our choice of $K$, we see that $(\ref{eq:+control})$
can be rewritten as
\begin{equation}\label{eq:+control-bis}
{\rm ord}_p(\#\Sha(E/K)[p^\infty])=2\cdot{\rm ord}_p([E(K):\bZ.{P}_K])
-\sum_{w\mid N}{\rm ord}_p(c_w(E/K)).
\end{equation}

On the other hand, as explained
in \cite[p.~47]{JSW} the Gross--Zagier formula \cite{GZ}, \cite{YZZ} (as refined in \cite{CSY})
can be paraphrased as the equality
\[
\frac{L'(E,1)}{\Omega_E\cdot{\rm Reg}(E/\bQ)}\cdot\frac{L(E^D,1)}{\Omega_{E^D}}=[E(K):\bZ.{P}_K]^2
%\cdot\prod_{\ell\mid N^-}c_\ell(E/K)
\]
up to a $p$-adic unit,\footnote{This uses a period relation coming from
\cite[Lem.~9.6]{skinner-zhang}, which assumes that $(D,pN)=1$, but
the same argument applies replacing $D$ by $D/(D,pN)$ in the last paragraph of the proof of their result.}
which combined with (\ref{eq:+control-bis}) and the immediate relation
\[
\sum_{w\mid N}c_w(E/K)
=\sum_{\ell\mid N}c_\ell(E/\bQ)+\sum_{\ell\mid N}c_\ell(E^D/\bQ)
\]
(see \cite[Cor.~7.2]{skinner-zhang}) leads to the equality
%\footnote{For this, if $p\mid N$ we need to clarify what happens with $c_v(E/K)$ for $v\mid p$.}
\[
{\rm ord}_p(\#\Sha(E/K)[p^\infty])={\rm ord}_p\biggl(\frac{L'(E,1)}{\Omega_E\cdot{\rm Reg}(E/\bQ)
\prod_{\ell\mid N}c_\ell(E/\bQ)}\cdot\frac{L(E^D,1)}{\Omega_{E^D}\prod_{\ell\mid N} c_\ell(E^D/\bQ)}\biggr).
\]

Finally, since $L(E^D,1)\neq 0$, by the known $p$-part of the Birch and Swinnerton-Dyer formula for $E^D$ (as recalled in
\cite[Thm.~7.2.1]{JSW}) we arrive at
\[
{\rm ord}_p(\#\Sha(E/\bQ)[p^\infty])={\rm ord}_p\biggl(\frac{L'(E,1)}{\Omega_E\cdot{\rm Reg}(E/\bQ)
\prod_{\ell\mid N}c_\ell(E/\bQ)}\biggr),
\]
concluding the proof of the theorem.
\end{proof}

%\begin{rem}
%The additional hypothesis in Theorem~A when $p\mid N$ seems essential to our method,
%since $p$ is throughout required to be split in $K$.
%However, the extension of the main results of \cite{bdp1} and \cite{wanIMC}
%to the case where $p$ is ramified in $K$ is a very interesting problem, and should allow
%one to dispense with the aforementioned hypothesis.
%\end{rem}

%\begin{appendix}

%\section{$p$-adic Waldspurger formula for multiplicative primes}

%\end{appendix}

\bibliographystyle{amsalpha}
\bibliography{Heegner}

\end{document}